\documentclass{amsart}[12pt]
\usepackage{amsmath,amssymb,pifont,calc,bm}

\newtheorem{theorem}{Theorem}[section]
\newtheorem{proposition}[theorem]{Proposition}

\newtheorem{corollary}[theorem]{Corollary}

\theoremstyle{definition}
\newtheorem{definition}[theorem]{Definition}

\theoremstyle{remark}
\newtheorem{remark}[theorem]{Remark}
\newtheorem*{remarks*}{Remarks}

\numberwithin{equation}{section}

\newcommand{\CC}{{\mathbb C}}

\newcommand{\ZZ}{{\mathbb Z}}

\newcommand{\PP}{{\mathbb P}}
\newcommand{\NN}{{\mathbb N}}
\newcommand{\BB}{{\mathbb B}}

\begin{document}
\openup3pt

\title[Limit currents]{Limit currents and value distribution of holomorphic maps}

\author{Daniel Burns}\address{Department of Mathematics \\
University of Michigan \\ Ann Arbor \\ MI \\ 48109 \\ USA}
\thanks{DB was supported in part by NSF grants DMS-0514070 and DMS-0805877. DB thanks the Universit\'e de Paris-Sud and the I.H.E.S. for support during the preparation of this paper.}
\email{dburns@umich.edu}

\author{Nessim Sibony}\address{D\'epartement Math\'ematiques \\Universit\'e Paris-Sud UMR 8628 \\91405 Orsay, France}
\email{nessim.sibony@math.u-psud.fr}
\date{}

\begin{abstract}
We construct $d$-closed and $dd^c$-closed positive currents associated to a holomorphic map $\phi$ via cluster points of normalized weighted truncated image currents. They are constructed using analogues of the Ahlfors length-area inequality in higher dimensions. Such classes of currents are also referred to as Ahlfors currents. We give some applications to equidistribution problems in value distribution theory.
\end{abstract}

\maketitle


\section{Introduction}
\label{sec:1}
Let $X$ be a complex manifold of dimension $k$, and $(Y, \omega)$ a compact k\"ahler manifold of dim $m \geq k$. We consider a non-degenerate holomorphic map $\phi: X \rightarrow Y$. We are interested in the distribution of pre-images of subvarieties of $Y$ under $\phi$. When $k = m = 1$, the theory is very well-developed, see, for example, \cite{wh}. In higher dimensions many questions remain open, but {\em cf.}, Griffiths \cite{pag}, Shabat \cite{bs}. 

We first construct some positive $d$-closed or $dd^c$-closed currents associated to $\phi$. When $X = \CC$, Ahlfors's length-area estimate implies that for appropriate subsequences $r_n \rightarrow +\infty$ the currents $\phi_*[D_{r_n}]/c_{r_n}$ cluster at positive closed currents of bidimension (1,1). Here $D_r$ is the disk of radius $r, [D_r]$ is the current of integration over this disk, and $c_r$ is a normalizing factor to guarantee mass 1. Such currents are useful in dynamics \cite{ds2}, \cite{ns} and value distribution theory \cite{mm}, for example. The present paper centers around extensions of Ahlfors' idea to higher dimensions, especially when $X$ is parabolic, or a bounded domain.

Let $\tau$ be a plurisubharmonic (p.s.h.) exhaustion function on $X$,
$$\tau: X \rightarrow [0, R), R > 0,$$
where $R$ could be finite. Recall that a manifold $X$ is parabolic if it admits a p.s.h. exhaustion function $\tau$ with $R = +\infty$, and such that $(dd^c \tau)^k$ vanishes outside a compact set. An example is $X = \CC^k, \tau = \log \| z \|$ outside a compact set. Riemann surfaces are parabolic if and only if they do not admit non-constant, bounded, subharmonic functions.

We consider positive currents $S_r = S_{j, r}$ of bidimension $(j, j)$ on $Y$ defined by
\begin{equation}
	\label{eqn:Stdef}
		<S_r, \varphi> = \frac1{c_r} \int_X u_r \, (dd^c \tau)^{k-j} \wedge \phi^*(\varphi)
\end{equation}
where $\varphi$ is a test form of bidegree $(j, j)$, and where we may take $u_r = (1-\frac{\tau}{r})^+$, a plurisuperharmonic function on $B_r := \{ z \in X \; | \; \tau(z) < r\}$ for all $r$. The constant $c_r = \int_X u_r (dd^c\tau)^{k-j}\phi^*(\omega^j)$ is a normalizing constant to have $S_r/c_r$ of  mass 1. When $X$ is parabolic of dimension $k$, for example, we show that for $j = 1$, the currents $S_r/c_r$ have at least one positive $d$-closed current among their cluster points.
\begin{definition} 
\label{def:tkdef}
Define the unaveraged characteristic function
\begin{equation}	
	\label{eqn:tkdef}
		t_{j}(r) = \int_{B_r} (dd^c \tau)^{k-j} \wedge \phi^*(\omega^{j}).
\end{equation}
\noindent The averaged characteristic function, or simply characteristic function, $T_j(r)$ is defined as
\begin{equation}
	\label{eqn:Tkdef}
		T_{j}(r) = \int_0^r \frac1s t_{j}(s) ds.
\end{equation}
\end{definition}
These characteristic functions are modeled on those of Nevannlina and later, in higher dimensions, of Chern, for example \cite{ssc}. 

Define the $d$-{\em mass ratio} $I_{j}(r)$ of degree $j$ as follows: 
\begin{equation} 
	\label{eqn:LAR}
		I_{j}(r) = \frac{\left(\int_0^r t_{j-1}(s) \, ds\right) \cdot t_{j}(r)}{(\int_0^r t_{j}(s) ds)^2}.
\end{equation}
\noindent See (\ref{eqn:defI}) and (\ref{eqn:Ibis}) for the origin and derivation of (\ref{eqn:LAR}).

We show, for arbitrary p.s.h. exhaustion $\tau$, that if $\lim_{r_{\ell} \rightarrow +\infty} I_{j}(r_{\ell}) = 0$, then all cluster points of the currents $S_{r_{\ell}}/c_{r_{\ell}}$ of bidimension $(j, j)$ are positive {\em closed} currents, cf. Theorem \ref{thm:1stdlimthm}. 

The question of the existence of closed currents for images of $\CC^k$ has been explored recently by de Th\'elin \cite{hdt}, where limit currents similar to those described above are called {\em Ahlfors currents}. The difference between \cite{hdt} and here is that we weight the integral in the definition of $S_r$ with the factor $u_r$, which is smoother than the characteristic function of $B_r$ used in \cite{hdt}. The condition guaranteeing the existence of closed limit currents seems more tractable than that of \cite{hdt} since it involves only the relative growth of $t_{j}$ and $t_{j-1}$, and not their derivatives, although \emph{cf.} Theorem 2 of \cite{hdt} on this point, and compare it to Theorem \ref{thm:egcase} below. Note that Theorem \ref{thm:egcase} produces $dd^c$-closed currents. Only the maximal dimensional case $j = k$, i.e., $S_r$ of bidimension $(k,k)$ is examined in \cite{hdt}, and only the case of $d$-closed limits. It turns out that for questions of value distribution, it can often be just as useful to find cluster points which are $dd^c$-closed, a situation to which we turn next. 

Assume $\tau = \log \sigma$, and redefine  $B_r = \{ \sigma < r\}$, and the current $S_r = S_{j,r}$ of bidimension $(j, j)$ in the $dd^c$-case as
\begin{equation}
	\label{eqn:Srdefddc}
	S_r(\psi) = \int_{B_r} \log^+ \frac{r}{\sigma} \, (dd^c \tau)^{k-j} \wedge \phi^*(\psi)
\end{equation}
for $\psi$ of bidegree $(j, j)$, with normalizing constant $c_r = S_r(\omega^j).$ 

Define the $dd^c$-{\em mass ratio} $J_{j}(r)$ of degree $j$ by
\begin{equation}
	\label{eqn:ddclar}
		J_{j}(r) = \frac{t_{j-1}(r)}{T_j(r)},
\end{equation}
where the denominator is the classical characteristic function (\ref{eqn:Tkdef}). Our main result is Theorem \ref{thm:ddc-closed}, which gives the following result.
\vskip 2mm
\noindent {\bf Main Theorem \ref{thm:ddc-closed}}
\emph{If $J_{j}(r_{\ell}) \rightarrow 0,$ then all cluster points of $S_{r_{\ell}}/c_{r_{\ell}}$ are $dd^c$-closed. Moreover, $\left<dd^c S_{r_{\ell}}/c_{r_{\ell}}, \psi\right> \rightarrow 0$ when $\psi$ is a bounded form.}
\vskip 2mm
\noindent In particular, we find conditions which ensure that there is a $dd^c$-closed current associated with a holomorphic map $\phi: \BB^k(1) \rightarrow Y$, with $\BB^k(1)$ the unit ball in $\CC^k$. A consequence of these conditions is the following Brody type result. 
\vskip 2mm
\noindent {\bf Theorem \ref{thm:Brody}} \; \emph{Let $\phi_n: \BB^k(1) \rightarrow Y$ be a sequence of holomorphic maps. Then either the graphs of the $\phi_n$ form a normal family of analytic sets, or there is a $j, 1 \leq j \leq k,$ and sequences $r_{\ell} \rightarrow R^-, n_{\ell} \rightarrow \infty$ such that $S_{j, r_{\ell}}/c_{j,r_{\ell}} = S_{\phi_{n_{\ell}}, j, r_{\ell}}/c_{n_{\ell},j,r_{\ell}}$ converges to a $dd^c$-closed current.}
\vskip 2mm

These results lead to several consequences in value distribution theory, and we record just one here, describing the value distribution of {\em points}. 
\vskip 2mm
\noindent {\bf Theorem \ref{thm:defvit0}} \emph{Let $\phi:\CC^k \rightarrow \PP^k$ be a non-degenerate holomorphic map. Assume that
\begin{equation}
	\label{equ:LARddck-1}
		\liminf_{r \rightarrow +\infty} \frac{t_{k-1}(r)}{T_k(r)} = 0.
\end{equation}
Then there exists a ``small" exceptional set $\mathcal{E}$ such that for $a \notin \mathcal{E}$, then 
\begin{equation}
	\label{eqn:vdpts}
		\limsup_{r \rightarrow +\infty} \frac{N(a, r)}{T_k(r)} = +1.
\end{equation}
In particular, the (2n-2+$\delta$)-dimensional Hausdorff measure of $\mathcal{E}$ is 0, for any $\delta > 0$}.
\vskip 2mm
Here $N(a, r)$ is the classical logarithmic average of the number of preimages of $a$ in the $\tau$-ball of radius $r$ ({\em c.f.} (\ref{eqn:NDa})), and $T_k(r) = \int_0^r t_k(s) \, \frac{ds}{s},$ the appropriate characteristic function for this dimension. The smallness of $\mathcal{E}$ is measured by a capacity, for which $\mathcal{E}$ is of capacity 0. In fact, we get for every codimension $j$ an exceptional set $\mathcal{E}_j$ of ``zero $j$-capacity" such that outside of $\mathcal{E}_j$ one has defect zero, in the sense of a dimension-appropriate case of a result similar to (\ref{eqn:vdpts}), provided that the appropriate $J_j(r)$ has $\liminf J_j(r) = 0.$ It seems that in previous work (see Shabat \cite{bs}, Griffiths-King \cite{gk}), the claim is that ``most" points are covered without a quantitative measure of the size of the defect locus. For analytic sets there are earlier results in this direction for the average growth of a hyperplane section, see Gruman \cite{lg}, Molzon-Shiffman-Sibony \cite{mss}.

Theorem \ref{thm:defvit0} and other results in section \ref{sec:appsvd} are sharper than stated here, since we give estimates for the rates of convergence. For these the second half of Theorem \ref{thm:ddc-closed} is crucial, and gives a formulation of the proximity term in the First Main Theorem of value distribution in our context, and an estimation in terms of mass ratios.

\vskip 3mm

Here is an outline of the paper. In section \ref{sec:d-closed} we estimate $\left<dS_r,\psi\right>/c_r$, and arrive at the $d$-mass ratios of degree $j$ as a useful bound. The rest of the section is devoted to estimating these mass ratios in concrete cases. The situation is especially clear when the domain $X$ of $\phi$ is parabolic, and when $X$ is furthermore of dimension one, our results are complete. 

Section \ref{sec:ddc-closed} is very analogous to section \ref{sec:d-closed}, but for $dd^c$-closed cluster currents. Of particular interest is the precise estimate in THeorem \ref{thm:ddc-closed},
$$|\left<dd^c S_{j,r}, \psi\right>/c_r| \leq C \|\psi\|_{\infty} \frac{t_{j-1}(r)}{T_j(r)},$$
valid for all bounded test forms $\psi$ of bidegree $(j,j)$, which is central in much of what follows, especially in section \ref{sec:appsvd}. The section closes with a theorem on the positivity of intersection of the cluster currents constructed in bidimension $(1,1)$ with analytic hypersurfaces which meet the image of $\phi$ non-trivially. This generalizes a result of McQuillan's for $X = \CC$ or a finite branched cover of $\CC$. Section \ref{sec:scaling} studies the effect of scaling on the estimates we use on $dd^cS_r$. In particular, because we can estimate $dd^cS_{j,r}$ for all

intermediate $j$, and not just $j = k$, we arrive at a ``multichotomy": either one of the $j, 1 \leq j \leq k$ gives rise to a positive, $dd^c$-closed limit current of $S_{j,r}/c_{j,r}$ or we get an estimate on the volume of the graph of $\phi$. This follows from the inductive structure of the various $dd^c$-mass ratios, and their relation to the mixed volumes calculation of the volumes of graphs in $X \times Y$. The Brody-type result described above follows. Section \ref{sec:appsvd} deals with the value distribution applications, and includes one corollary about the behavior of leaves of singular holomorphic foliations of $\PP^m$. Section \ref{sec:hdims} examines the size of the set of limit currents constructed here using results in complex dynamics. The result is a kind of higher dimensional equidistribution according as a limit current is unique. The final section \ref{sec:growth} relates the mass ratio conditions which this article is based on to a couple of examples of classical order of growth conditions, such as finite order,  on maps $\phi$.

\vskip 2mm

\begin{remark}
	\label{rmk:flex}
	In what follows, we will have considerable flexibility in how we construct the limit currents. There are at least two forms of growth measurements one might use, depending on whether one uses averaged or unaveraged characteristic functions. The averaged functions arise when one averages out the currents $S_r$ via
$$\tilde{S}_r(\psi) = \int_0^r \frac{ds}{s} \int_{B_s} u_s (dd^c \tau)^{k-j} \wedge \phi^*(\psi),$$
for test forms $\psi$ of bidegree $(j,j)$, where the only difference between the $d$-case and the $dd^c$-case is in the choice of $u_r$ as above. In practice, there are only minor technical differences in these cases, and we content ourselves with mentioning the averaged currents here and in remark \ref{rmk:int} below. The differences in arguments between the $d$-closed limits and the $dd^c$-closed limits are more substantial, and we carry out more or less parallel arguments in these two cases in sections \ref{sec:d-closed}  and \ref{sec:ddc-closed}, respectively. The $dd^c$ case requires a regularization of $u_r$.
\end{remark}



\section{First limits: $d$-closed currents}
\label{sec:d-closed}

Let $X$ be a complex manifold of dimension $k$, and $(Y, \omega)$ a compact K\"ahler manifold of dimension $m \geq k$. We assume $X, Y$ conected. Let $\phi: X \rightarrow Y$ be a non-degenerate holomorphic map, \emph{i.e.}, the rank of $d\phi(x_0) = k$ at some $x_0 \in X$. Let $\tau: X \rightarrow [0, R), 0 < R \leq +\infty$ be a smooth plurisubharmonic exhaustion function. Set $B_r = \{x \mid \tau(x) \leq r\}$, which is compact for $r < R$. For convenience, we will usually assume that 
\begin{equation}
	\label{eqn:convention}
	\tau \geq r_0 > 0.
\end{equation}

Let $u_r$ be a family of continuous positive plurisuperharmonic functions on $B_r$, $r \in [0, R)$. We consider the family of positive currents of bidimension $(j,j)$ on $Y$ defined by
   \begin{equation}
   \label{eqn:defT}
  S_r(\psi) = S_{j,r}(\psi) =  \int_{B_r} u_r (dd^c \tau)^{k-j} \wedge \phi^{*}(\psi),
    \end{equation}
where $\psi$ is a smooth test form of bidegree $(j,j)$ on $Y$, and set $c_r = c_{j,r} = S_{j,r}(\omega^j)$. We will study the cluster points of the family of normalized positive currents $S_r(\cdot)/c_r$ of mass 1.
Different choices of $u_r$ will prove useful in what follows. In this section we consider cases where the proper choice of $u_r$, and suitable conditions on $\phi, \tau, \omega$,  lead to $d$-closed currents as cluster points of the normalized $S_r$'s.

In particular, we will work mainly in this section with $u_r := (1 - \frac{\tau}{r})^+ = \chi(v_r),$ where $v_r = 1 - \frac{\tau}{r},$ and $\chi = \max(t, 0).$ We want to find conditions which guarantee that $dS_{r_{\ell}}/c_{r_{\ell}} \rightarrow 0,$ for suitable sequences $r_{\ell} \rightarrow R$. For this it is enough to estimate $dS_r$ on test forms of the type $\psi = \theta \wedge \beta^{k-1},$ with $\theta$ a (1,0)-form and $\beta$ an arbitrary k\"ahler form. This is because we can first assume $\psi$ has components only in bidegrees $(j,j-1)$ and $(j-1,j)$, and is real, and because secondly any such $\psi$ can can be written as a finite sum (with an {\em a priori} bounded number of terms),
\begin{equation}
	\label{eqn:decomp}
	\psi = \sum_{\nu = 1}^N \theta_{\nu} \wedge \beta_{\nu}^{j-1} + \sum_{\nu = 1}^N \overline{\theta}_{\nu} \wedge \beta_{\nu}^{j-1},
\end{equation}
where $\theta_{\nu}, \beta_{\nu}$ are as claimed. We note that this can be done  in such a way that 
\begin{equation}
	\label{eqn:decompbounds}
	\begin{array}{l}
	\frac{i}{2} \theta_{\nu} \wedge \overline{\theta}_{\nu} \leq C \, \|\psi\|_{\infty}^2 \;  \omega, \; \text{and} \\
		\\
	0 \leq \beta_{\nu} \leq \omega, \nu = 1,\ldots, N.
	\end{array}
\end{equation}

	By the Schwarz inequality, we get
   \begin{equation}
  	 \label{eqn:SchwS}
  		 \aligned
 		 \left|  \left<dS_r, \psi \right> \right|  & = \, \left| \int_X \chi'(v_r) \, dv_r \wedge (dd^c \tau)^{k-j} \wedge  \phi^{*}(\theta) \wedge \phi^{*}(\beta^{j-1}) \right| \\
  		& \leq \left( \int_{B_r} \chi'(v_r) , dv_r \wedge d^cv_r \wedge (dd^c \tau)^{k-j} \wedge \phi^{*}(\beta^{j-1})\right)^{\frac12} \\
  		& \times \left( \int_{B_r} \chi'(v_r) \phi^{*}(\theta) \wedge \phi^{*}(\bar{\theta})\wedge (dd^c \tau)^{k-j} \wedge \phi^{*}(\beta^{j-1}) \right)^{\frac12}.
  		\endaligned
   \end{equation}
It follows that 
  \begin{equation}
  \label{eqn:estdS}
  \aligned
  \left| \left<dS_r, \psi \right>\right| 
  & \leq \, C \; \|\psi\|_{\infty} \left(\int_{B_r} dv_r \wedge d^cv _r \wedge (dd^c \tau)^{k-j} \wedge \phi^{*}(\omega^{j-1})\right)^{\frac12} \\
  & \times \; \left(\int_{B_r} (dd^c\tau)^{k-j} \wedge \phi^{*}(\omega^j)\right)^{\frac12}.
  \endaligned
  \end{equation}
Hence, one has
  \begin{equation}
  	\label{eqn:massratios}
  	\begin{array}{rcl}
	\left|\frac{\left<dS_r, \psi \right>}{c_r}\right|^2 & \leq & C \, \|\psi\|^2_{\infty} \\ 
		&	&	\\
		&	& \times \;\; \frac{(\int_{B_r} dv_r \wedge d^cv_r \wedge (dd^c \tau)^{k-j} \wedge \phi^{*}(\omega^{j-1}))(\int_{B_r} (dd^c \tau)^{k-j} \wedge \phi^{*}(\omega^j))}{(\int_{B_t} u_t \, (dd^c \tau)^{k-j} \wedge \phi^{*}(\omega^j))^2}.
	\end{array}
  \end{equation}
\remark
\label{ref:sings}
With small technical modifications, we can allow $X$ to be a singular analytic space.

We formalize this condition. First set ${I}_j(r)$ equal to (the essential part of) the right hand side of (\ref{eqn:massratios}), that is,
\begin{equation}
	\label{eqn:defI}
	I_j(r) = \frac{(\int_{B_r} dv_r \wedge d^cv_r \wedge (dd^c \tau)^{k-j} \wedge \phi^{*}(\omega^{j-1}))(\int_{B_r} (dd^c \tau)^{k-j} \wedge \phi^{*}(\omega^j))}{(\int_{B_r} u_r \, (dd^c \tau)^{k-j} \wedge \phi^{*}(\omega^j))^2}.
\end{equation}
We have proved the following basic theorem.
\begin{theorem}
	\label{thm:1stdlimthm}
		If the exists a sequence $r_{\ell} \rightarrow \infty$ such that $I_j(r_{\ell}) \rightarrow 0$, then any limit current of $S_{r_{\ell}}/c_{r_{\ell}}$ 
		is a closed and positive current of mass 1. Moreover, $\lim_{r_{\ell} \rightarrow \infty} \frac1{c_{r_{\ell}}} \left<dS_{r_{\ell}}, \psi\right> = 0,$ for any bounded test form $\psi$ of degree $2j-1$.
\end{theorem}
We are thus led to study the ratios $I(r) = I_j(r)$ of (\ref{eqn:defI}).
Let us introduce characteristic functions appropriate to all dimensions as in (\ref{eqn:tkdef}) and (\ref{eqn:Tkdef}) above. Similar notions have been used in the holomorphic dynamics literature under the name of dynamical degrees: when $f$ is a meromorphic self-map of a compact K\"ahler manifold $Y$ of dimension $k$, then the $j$-th dynamical degree $\lambda_j$ is defined as
\begin{equation}
	\label{eqn:dddef}
	 \lambda_j := \lim_{n\rightarrow \infty} \, (\int_Y \omega^{k-j} \wedge (f^n)^*(\omega^j))^{1/n},
\end{equation}
see, for example, \cite{ds2} for references.
\begin{definition}
	\label{def:DD}
	For $0 \leq j \leq k$, set $t_j(r) = \int_{B_r} (dd^c\tau)^{k-j} \wedge \phi^{*}(\omega^j).$ 
\end{definition}
 \noindent We express the components of the $I_j(t)$'s in terms of these $t_j$'s. We write out the case of $j = k$ only; the others are similar. First 
  \begin{equation}
 	 \label{eqn:lambints}
 		 \aligned
 			 \int_{B_r} dv_r \wedge d^c v_r \wedge \phi^{*}(\omega^{k-1}) 
 				 & = \frac{1}{r^2} \, \left( \int_{\partial B_r} \tau \, d^c\tau \wedge \phi^{*}(\omega^{k-1})- \int_{B_r} \tau \, dd^c\tau \wedge \phi^{*}(\omega^{k-1}) \right) \\ 
 				 & = \frac{1}{r^2} \left(r \int_{B_r} dd^c\tau \wedge \phi^{*}(\omega^{k-1})  - \int_{B_r} \tau \, dd^c\tau \wedge \phi^{*}(\omega^{k-1}) \right) \\
 				 & = \frac{1}{r^2} \, \int_{B_r} (r-\tau) \, dd^c\tau \wedge \phi^{*}(\omega^{k-1}) \\
 				 & =  \frac{1}{r} \int_{B_r} (1-\frac{\tau}{r}) \, dd^c\tau \wedge \phi^{*}(\omega^{k-1}) \\
				  & = \frac1r \int_{B_r} (\int_0^{1-\frac{\tau}{r}} ds) \, dd^c\tau \wedge \phi^{*}(\omega^{k-1}) \\
				  & = \frac1r \int_0^1 ds \int_{B_{r(1-s)}} dd^c\tau \wedge \phi^{*}(\omega^{k-1}) \\
				  & = \frac{1}{r^2} \int_0^r ds \, \int_{B_s} dd^c \tau \wedge \phi^{*} (\omega^{k-1}) \\
				  & = \frac1{r^2} \int_0^r t_{k-1}(s) \; ds
		  \endaligned
 \end{equation}
Similarly, 
$$\int_{B_r} u_r  \, \phi^{*}(\omega^k) = \frac{1}{r} \int_0^r ds \int_{B_s} \phi^{*}(\omega^k) = \frac{1}{r} \int_0^r t_k(s) \, ds.$$
Thus we can re-express $I(r)$ as
\begin{equation}
	\label{eqn:Ibis}
	I(r) = \frac{(\int_0^r t_{k-1}(s) \, ds) \; t_k(r)}{(\int_0^r t_k(s) \,ds)^2}.
\end{equation}
With (\ref{eqn:Ibis}) in hand, we can express relatively natural conditions on the growth or decay of ratios of volumes, similar in spirit to the original Ahlfors conditions, which guarantee that $I(r_{\ell}) \rightarrow 0$ along some suitable sequences $r_{\ell} \rightarrow \infty$. For convenience, set $c_{r_{\ell}} = c_{\ell}$ below.
\begin{theorem}
	\label{thm:dclosedMT}
	Let $\phi, X, Y, \tau$ be as above.
	\vskip 2mm
	\noindent 1. Assume $R = \infty$, and that 
	\begin{equation}
		\label{eqn:simpledMR}
		\lim \frac{t_{j-1}(r)}{\int_0^r t_j(s) \, ds} = 0.
	\end{equation}
	Then there is a sequence $r_{\ell} \rightarrow \infty$ such that $S_{r_{\ell}}/c_{\ell}$ converges to a positive closed current. Moreover, $\left<dS_{r_{\ell}}/c_{\ell}, \psi \right>  \rightarrow 0$, for any bounded test form $\psi$.
	\vskip 2mm
	\noindent 2. Assume $R = \infty$, and let $\alpha = \alpha(s)$ be a continuous function such that $\int_0^{\infty} \frac{ds}{\alpha(s)} = \infty$. Assume further that
	\begin{equation}
		\label{eqn:alphaMR}
		\liminf \; \frac{\alpha(\int_0^r t_j(s) \, ds) \cdot \int_0^r t_{j-1}(s) ds }{(\int_0^r t_j(s) \, ds)^2} = 0.
	\end{equation}
	Then there is a sequence $r_{\ell} \rightarrow \infty$ such that $S_{r_{\ell}}/c_{\ell}$ converges to a positive closed current. Moreover, $\left<dS_{r_{\ell}}/c_{\ell}, \psi \right>  \rightarrow 0$, for any bounded test form $\psi$.
	\vskip 2mm
	\noindent 3. Assume $R < \infty$, and that
	\begin{equation}
		\label{eqn:minimaldMR}
			\int_{r_0}^R \frac{dr}{\int_{r_0}^r t_{j-1}(s) \, ds} = \infty.
	\end{equation}
	Then there is a sequence $r_{\ell} \rightarrow \infty$ such that $S_{r_{\ell}}/c_{\ell}$ converges to a positive closed current. Moreover, $\left<dS_{r_{\ell}}/c_{\ell}, \psi \right>  \rightarrow 0$, for any bounded test form $\psi$.
\end{theorem}

\begin{proof}
We write out the case $j = k$; the others proceed similarly. For notational simplicity, set $A(r) = \int_0^r t_k(s) \, ds$. We see that $I(r) \geq c$ is equivalent to\footnote{One might have to assume $r \geq$ some $r_1$ to guarantee $A(r) \neq 0$ in the arguments below. We will assume, WLOG, that $r_1 = 0$.}
\begin{equation}
	\label{eqn:Aineq1}
	c \leq \frac{A'(r)}{A^2(r)} \, \int_0^r t_{k-1}(s) \, ds.
\end{equation}
We show that (\ref{eqn:Aineq1}) contradicts, in turn, each of the three hypotheses in the statement of Theorem \ref{thm:dclosedMT}. The final comments about convergence for bounded test forms follow directly from (\ref{eqn:massratios}) and (\ref{eqn:Ibis}).
\vskip 2mm
\noindent 1. We integrate (\ref{eqn:Aineq1}) from 1 to $r$ and get
\begin{equation}
	\label{eqn:Aineq2}
		\begin{array}{rcl}
		c(r - 1) & \leq & \left[ -\frac1{A(t)} \, \int_0^t t_{k-1}(s) \, ds \right]_1^r + \int_1^r \frac{t_{k-1}(s)}{A(s)} \, ds \\
			&	&	\\
			& \leq & \int_1^r \frac{t_{k-1}(s)}{A(s)} \, ds + O(1),
		\end{array}
\end{equation}
on suitable sequences of $r \rightarrow R$. If $\lim \frac{t_{k-1}(r)}{A(r)} = 0,$ we get a contradiction for some $R >> 0$.
\vskip 2mm
\noindent 2. Recall that since $A$ is increasing, then $A'(r) \leq \alpha(A(r))$ outside a set $E$ of finite length. If $E = \{r \, | \, A'(r) > \alpha(A(r))\}$, one has 
$$\text{measure}(E) \leq \int_E \frac{A'}{\alpha(A)} \, dr \leq \int_0^{\infty} \frac1{\alpha(u)} \, du < \infty.$$
From (\ref{eqn:Aineq1}) we get that on the complement of $E$
$$c\leq \frac{\alpha(A)}{A^2} \, \int_0^r t_{k-1}(s) \, ds,$$
which is a contradiction.
\vskip 2mm
\noindent 3. If $c \leq \frac{A'}{A^2} \int_0^r \, t_{k-1}(s) \, ds,$ with $r > R_0$, then
$$\int_{r_0}^r \frac{c \, ds}{\int_0^r t_{k-1} \, ds} \leq \int_{r_0}^r \frac{A'}{A^2} \,ds = \left[ -\frac{1}{A} \right]_{r_0}^r \leq \frac{1}{A(R_0)} < \infty,$$ which leads again to a contradiction and proves 3.
\end{proof}

We examine next another case where we can analyze the condition $I(r_{\ell}) \rightarrow 0$ by manipulation of ratios of volume growth. We start from the simple observation that (\ref{eqn:Aineq1}) is equivalent to
\begin{equation}
	\label{eqn:MR1dibis}
		\frac1c \frac{A'}{A^{1+\delta}} \geq \frac{A^{1-\delta}}{\int_0^r t_{k-1}(s) \, ds}, \text{for any} \; \delta > 0.
\end{equation}
Integrating (\ref{eqn:MR1dibis}) on $[r_0, r]$, one gets
\begin{equation}
	\label{eqn:MR1dint}
		\frac1c \left[ -\frac{A^{-\delta}}{\delta}\right]_{r_0}^r = \frac1{c \, \delta} \left[ A^{-\delta}(r_0) - A^{-\delta}(r) \right] \geq \int_{r_0}^r \frac{A^{1-\delta}(t)}{\int_0^t t_{k-1}(s) \, ds} \; dt.
\end{equation}
We conclude that
\begin{equation}
	\label{eqn:MR1dibdd}
		\frac1c A^{-\delta}(r_0) \geq \delta \int_{r_0}^r \frac{A^{1-\delta}(t)}{\int_0^t t_{k-1}(s) \, ds} \; dt, \text{for any } \delta > 0.
\end{equation}
In particular, if
\begin{equation}
	\label{eqn:MR1supdelta}
			\sup_{\delta > 0, r < R} \; \delta  \; \int_{r_0}^r \frac{A^{1-\delta}(t)}{\int_0^t t_{k-1}(s) \, ds} \, dt = \sup_{\delta > 0, r < R} \;  \int_{r_0}^r \frac{\, (\int_0^t t_k(s) \, ds)^{1-\delta}\,}{\int_0^t t_{k-1}(s) \, ds} \, dt= +\infty,
\end{equation}
then inequality (\ref{eqn:MR1dibdd}) fails for some $\delta > 0, r \in (0,R)$. Since $c > 0$ was arbitrary, we obtain the following corollary of Theorem \ref{thm:1stdlimthm}.

\begin{corollary}
	\label{cor:MR1delta}
	If (\ref{eqn:MR1supdelta}) holds, then there are closed, positive currents $S$ among the cluster points of $S_r/c_r.$
\end{corollary}

Focusing next on the case $\delta = 1$ in (\ref{eqn:MR1dibdd}), if
\begin{equation*}
	\int_{R_0}^R \frac1{\int_0^t t_{k-1}(s) \, ds} \, dt = +\infty,
\end{equation*}
then we can apply corollary \ref{cor:MR1delta}. If furthermore $k = 1$, this last becomes
\begin{equation}
	\label{eqn:MR1dibddk1}
	\int_{R_0}^R \frac1{\int_0^t t_0(s) \, ds} \, dt = +\infty,
\end{equation}
a condition which is interesting since it is {\emph{independent of}} $\phi$. Note that this condition can also be used for $j = 1$ and arbitrary $k$ to construct closed limit currents of bidimension $(1,1)$. Therefore, as a special case, we have the following corollary.

\begin{corollary}
	\label{cor:k1par}
		If $\dim X = k, R = \infty$ and $\tau$ is a parabolic exhaustion of $X$, then $\phi$ admits closed positive limit currents of bidimension $(1,1)$ as limit points of $S_{1,r}/c_{1,r}$, for any $Y, \phi$ and $\omega$.
\end{corollary}

\begin{proof} 
If $\tau$ is a parabolic exhaustion, i.e., $(dd^c \tau)^k = 0$,  for $\tau \geq \text{some} \; r_0$, then
\begin{equation*}
	\aligned
	t_0(r) & = \int_0^r (dd^c \tau)^k \\
	& = \int_{\{r_0 < \tau < r\}} (dd^c \tau)^k + \int_{B_{r_0}} (dd^c \tau)^k \\
	& =  C 
	\endaligned
\end{equation*}
for $r \geq r_0 >> 0.$ In particular, $ \int_0^r \frac{dt}{\int_0^t t_0(s) \, ds}$ diverges logarithmically, verifying condition (\ref{eqn:MR1dibddk1}). 
\end{proof}
Examining the proof of corollary \ref{cor:k1par} shows the conclusions to hold whenever (\ref{eqn:MR1dibddk1}) is verified, and the corollary lets one interperet (\ref{eqn:MR1dibddk1}) as a weak form of parabolicity for the pair $X, \tau$, since it is independent of $\phi, Y, \omega$. Along the same lines, suppose that the denominator $\int_0^r t_{k-1}(s) \, ds$ of the integrand of (\ref{eqn:MR1supdelta}) is bounded, but that $A(r)$ is unbounded (as in the parabolic case), then for $\delta \in (0, 1)$, (\ref{eqn:MR1dibdd}) gives
\begin{equation}
	\label{eqn:pseudopar}
	\frac1c A^{-\delta}(r_0) \geq \delta \int_{r_0}^r \frac{A^{1-\delta}(t)}{\int_0^t t_{k-1}(s) \, ds} \; dt > c' \int_{r_0}^r A^{1-\delta}(t) \, dt,
\end{equation}
a contradiction, if $\int_{r_0}^r A^{1-\delta}\,dt$ is unbounded. 

The same considerations apply to a bounded situation as follows. Let $\phi_n: \triangle \rightarrow Y$ be a sequence of maps from the unit disk to $Y$. 

\begin{proposition}
	\label{prop:nfinrad}
	Let $\phi_n$ be a sequence of maps from $\triangle$ to $Y$. Assume that $A_{n, \delta} := \int_0^r \left(\int_{D_s} \phi^*_n(\omega)\right)^{1-\delta} \, ds, 0 < r, \delta < 1,$ is unbounded. Then there is a positive, d-closed current of bidimension (1,1) among the cluster currents of $S_{1,r,n}/c_{1,n,r}.$
\end{proposition}

The situation for $X$ of dimension $k=1$ and $R = \infty$ divides very neatly by corollary \ref{cor:k1par} into two cases, according as $\int_X \phi^{*} \omega < +\infty$ or $\int_X \phi^{*} \omega = +\infty$.

\begin{corollary}
\label{cor:curvesfinite}
In corollary \ref{cor:k1par}, if $\int_X \phi^{*} \omega$ is finite, then the currents $S_r/c_r$ converge weakly to the current $S(\varphi) := \int_X \phi^{*}(\varphi)/\int_X \phi^{*} \omega$. 
\end{corollary}

\begin{proof}  Write $S_r(\omega)$ as
\begin{equation*}
\aligned
\int_{B_r} \chi(v_r) \, \phi^* \omega & =  \int_X \phi^* \omega - \int_X (\chi(v_r) - 1) \phi^* \omega 
\endaligned
\end{equation*}
where $\lim_{r \rightarrow +\infty}  \int_X (\chi(v_r) - 1) \phi^* \omega = 0$, by dominated convergence. The same observation applied to $S_r(\varphi)$ gives the corollary.
\end{proof}

Notice, however, that $\int_X \phi^*\omega$ unbounded does not imply the existence of a positive closed cluster current if $X$ is not parabolic. For example, a generic (singular) holomorphic foliation $\mathcal{F}$ of $\PP^2$ does not have a directed positive closed current even though all leaves of $\mathcal{F}$ have infinite area. See \cite{fs} for details. 

Recall that a Riemann surface is parabolic if there is no non-constant bounded subharmonic function on it, equivalently, if it does not admit a Green's function. (\cite{as}, p. 204). Thus, in the case of the generic foliation $\mathcal{F}$ of $\PP^2$, for example, the non-existence of directed positive closed currents implies by Corollary \ref{cor:k1par} that all leaves must admit non-trivial bounded subharmonic functions and must admit Green's functions.

\remark
\label{rem:parabolic}
In the situation of corollary \ref{cor:curvesfinite} when $X$ is an open Riemann surface with a parabolic exhaustion function in the sense of Stoll \cite{ws}, that is, when the exhaution $\log \tau$ is harmonic and has no critical points outside a compact set, then $X$ can be compactified to $\bar{X}$ by adding a finite number of points at infinity, and if the area of $\phi(X)$ is finite,  the mapping $\phi$ can be extended across these finitely many points. It suffices to observe that the graph of $\phi$ has finite area, and hence Bishop's extension theorem \cite{eb} says that its closure is an analytic set. In this case the current $S$ of corollary \ref{cor:curvesfinite} is given by integrating over the image $\phi(\bar{X})$, counting multiplicities.

\begin{corollary}
\label{cor:curvesinfinite}
In corollary \ref{cor:k1par}, if $\int_X \phi^* \omega = \infty,$ then the support of $S$ is contained in the intersection $\cap_{r \geq r_0} \overline{\phi(X \setminus B_r)}$.
\end{corollary}

\begin{proof}
Fix any $r_1 \geq r_0$, and write, for $r > r_1,$ 
\begin{equation}
\label{eqn:curvesinfpf}
\aligned
S_r(\psi) & = \int_{B_r} \chi(v_r) \phi^* \psi \\
& = \int_{B_{r_1}} \chi(v_t) \, \phi^* \psi + \int_{B_r \setminus B_{r_1}} \chi(v_t) \, \phi^* \psi \\
& = O(1) + \int_{B_r \setminus B_{r_1}} \chi(v_r) \, \phi^* \psi 
\endaligned
\end{equation}
If $\int_X \phi^*(\omega) = \infty$, this last shows that any cluster point of $S_r(\cdot)/S_r(\omega)$ is supported in $\overline{\phi(X \setminus B_{r_1})}$. Since $r_1$ was arbitrary, the result follows.
\end{proof}

\begin{remark}
\label{rem:localMR}
We can localize these arguments in dimension 1 as follows. Let $\Delta_{\rho}^* = \{z\in \CC \mid 0 < |z| < \rho\}$. Replace $X$ above by the punctured disk $\Delta^* = \Delta_1^*$ and take $v_t = \frac1t \log \frac1{\| z \|^2}$, which is a parabolic exhaustion. Applying the arguments above directly to a holomorphic map $\phi: \Delta^* \rightarrow Y$, we arrive at the dichotomy: for $\rho \in (0,1)$, either $\int_{\Delta_{\rho}^*} \phi^* \omega < +\infty$, and $\phi$ has a meromorphic extension across $0 \in \Delta$, or $\int_{\Delta_{\rho}^*} \phi^* \omega = +\infty$ and there is a closed, positive current $S$ on $Y$ with support contained in $\cap_{\rho \in (0,1)} \overline{\phi(\Delta_{\rho}^*)}$. If $\dim Y = 1$, this implies, in particular, the classical Casorati-Weierstrass theorem, but is sharper, since the identification of the limit current in the equidimensional case with $\frac1c [Y]$ gives a result on the equidistribution of values.
\end{remark}

To be more precise about the last remark, make a definition.
\begin{definition}
	\label{def:pod}
	A point $p \in Y$ is a $\phi$-density point if for every $\delta > 0$ there is a constant $\kappa_{\delta} > 0$ such that
	\begin{equation}
		\label{eqn:podineq}
		\liminf_{r \rightarrow R^-} \frac{\int_{B_r} \phi^*(\chi_{B_{\delta}(p)} \omega^k)}{\int_{B_r} \phi^*(\omega^k)} \geq \kappa_{\delta}.
	\end{equation}
\end{definition}
A point $p \in Y$ is a $\phi$-density point if and only if it is in the support of a cluster current of the family $S_r/c_r.$ The case $j=k=m$ of Theorem \ref{thm:1stdlimthm} then has the conclusion that every $p \in Y$ is a $\phi$-limit point, which adds some quantitative refinement to the mere density of $\phi(X)$.

It is natural in the present context to consider the closed set of all the positive closed currents which arise by the construction above. 
\begin{definition}
	\label{def:spacedclosed}
	Let $\mathcal{C}_j(\phi)$ denote the space of all positive closed currents of bidimension $(j,j)$ on $Y$ which are cluster points of currents of 	the form $S_{j,r}/c_{j,r}$ associated to $\phi$. 
\end{definition}
In Section \ref{sec:hdims} below we consider one case where $\mathcal{C}_j$ is shown to consist of one element using results from complex dynamics.

\begin{remark}
	\label{rmk:flex}
	In principle, of course, functional manipulations of (\ref{eqn:Aineq1}) other than (\ref{eqn:MR1dibis}) and following can be made which might lead to interesting conditions on $\phi$ for producing closed positive currents among the limit points of $S_r/c_r$. Other simple forms of $u_r$ as at the top of this section, or in remark \ref{rem:localMR} above, are useful for producing other kinds of limit currents. In section \ref{sec:ddc-closed} below we consider mainly the case of $dd^c$-closed limit currents, but also one case of $d$-closed currents, in Theorem \ref{thm:logdclosed}.
\end{remark}


\section{Limit currents which are  $dd^c$-closed}
\label{sec:ddc-closed}

In this section we take weighting functions much as in section \ref{sec:d-closed} above, but which lead to $dd^c$-closed currents of bidimension $(j,j), 1 \leq j \leq k = \dim X$. In many cases these can be as useful as the closed currents of section \ref{sec:d-closed} above, and in the equidimensional case $j = k = m = \dim Y$, they are equivalent. 

Assume now that $\log \sigma$ is a plurisubharmonic exhaustion of $X$, set $v_r = \log \frac r{\sigma}, u_r = \chi(v_r) = \log^+ \frac r{\sigma},$ where $\chi = \max(t,0)$. As in (\ref{eqn:Srdefddc}) we set 
\begin{equation}
	\label{eqn:defSddc}
	S_r(\psi) = \int_{B_r} u_r (dd^c \log \sigma)^{k-j} \phi^{*}(\psi),
\end{equation}
where $\psi$ is a test $(j,j)$-form on $Y$ and $B_r := \{ \sigma < r\}$, and define the $dd^c$-mass ration $J_j(r)$ by 
\begin{equation}
\label{eqn:defddcMR}
J_j(r) := \frac{ \int_{B_r} (dd^c \log \sigma)^{k-j+1} \wedge \phi^{*}(\omega^{j-1})}{\int_{B_r} u_r(dd^c\log \sigma)^{k-j} \wedge \phi^{*}(\omega^j)}.
\end{equation}
\begin{definition}
	\label{def:ddcMR}
	We say that $\phi, \sigma, \omega$ \emph{satisfy condition $dd^c$-MR} if 
	\begin{equation}
		\label{eqn:ddcMR0}
		\liminf J_j(r) = 0.
	\end{equation}
\end{definition}

As in (\ref{eqn:tkdef}), call the numerator in (\ref{eqn:defddcMR}) $t_{j-1}(r)$ and the denominator $T_j(r)$. Thus $T_j(r) = c_{j,r} = \int_{B_r} \log^+ \frac{r}{\sigma} \,(dd^c \log \sigma)^{k-j} \wedge \phi^*(\omega^k) = \int_0^r \frac1s t_j(s) \, ds,$ as in (\ref{eqn:Tkdef}), and so things simplify to:
\begin{equation}
	\label{eqn:defddcMRbis}
	J_j(r) = \frac{t_{j-1}(r)}{T_j(r)}.
\end{equation}
\begin{theorem}
\label{thm:ddc-closed}
Suppose $\phi, \sigma, \omega, r$ satisfy condition $dd^c$-MR, and $Y$ is a compact k\"ahler manifold. Then any cluster point $S_{\infty}$ of $S_r/c_r$ is a positive $dd^c$-closed current supported on $\overline{\phi(X)}$. Furthermore, 
\begin{equation}
	\label{eqn:vitesseddc}	
	\frac{1}{c_r} |\left<dd^c S_r, \psi\right>| \leq C \|\psi\|_{\infty} \, \frac{t_{j-1}(r)}{T_j(r)}, 
\end{equation}
for any bounded test form $\psi$ of bidegree $(j-1, j-1)$, where the constant $C > 0$ is independent of $r, \psi, \phi$.
\end{theorem}
\begin{proof}
We would like to get estimates on $\frac1{c_r} \left<dd^c S_r, \psi \right>$. To do so, we will first smooth out the function $\chi$. For $r > 0,$ let $v_r = \log \frac{r}{\sigma},$ and for each $\delta > 0$, let $u_{\delta, r} = \chi_{\delta}(v_r)$, where $\chi_{\delta}$ is a convex, increasing function which is $\equiv 0,$ on $(-\infty, 0)$, and $\chi''(s) = \frac1{\delta} \chi_{[0, \delta]},$ where $\chi_{[0,\delta]}$ is the characteristic function of $[0,\delta]$, and $\delta$ will tend to 0 later. 
We write the proof out only in the case $j= k$, the others being completely similar. We set $c_{k,r} = c_r,$ and suppress the index $\delta$ on $\chi_{\delta}$ for the moment.
\begin{equation}
\label{eqn:ddcSt}
\aligned
\frac1{c_r} \left<dd^cS_r, \beta \right> &= \frac1{c_r} \int_X dd^c(\chi(v_r)) \wedge \phi^{*}(\psi) \\
&= \frac1{c_r} \int_X (\chi'(v_r) dd^c v_r  \wedge \phi^{*}(\psi) + \chi''(v_r) dv_r \wedge d^c v_r) \wedge \phi^{*}(\psi)  \\
&= \frac1{c_r} \int_X -\chi'(v_r) dd^c \log \sigma \wedge \phi^{*}(\psi)  \\
& \; \; \; \; + \frac1{c_r} \int_X \chi''(v_r) d\log \sigma \wedge d^c\log \sigma \wedge \phi^{*}(\psi) \\
&= I_1 + I_2.
\endaligned
\end{equation}

Looking first at $I_1$, we remark that there is a constant $C > 0$, independent of $\psi, \phi$ such that $|\phi^*(\psi)| \leq C \|\psi\|_{\infty} \,\phi^*(\omega^{k-1}).$ Since $0 \leq \chi' \leq 1$, we get
\begin{equation}
	\label{eqn:I1}
	|I_1| \leq C \|\psi\|_{\infty} \frac1{c_r} \, \int_{B_r} dd^c \log \sigma \wedge \phi^*(\omega^{k-1}) = C \|\psi\|_{\infty} \frac{t_{k-1}(r)}{T_k(r)}.
\end{equation}

Passing to $I_2$, we see
\begin{equation}
	\label{eqn:I2}
	\begin{array}{rcl}
	T_k(r) |I_2| & = & \frac{1}{\delta} \; \, |\int_{\{r < \sigma < r + \delta\}} d\log \sigma \wedge d^c \log \sigma \wedge \phi^*(\psi) | \\
			&	&	\\
			  & \leq & \frac{C \|\psi\|_{\infty}}{\delta} \int_{\{r < \sigma < r + \delta\}} d\log \sigma \wedge d^c\log \sigma \wedge \phi^*(\omega^{k-1}) \\
			  &	&	\\
			  & = & \frac{C \|\psi\|_{\infty}}{\delta} \int_{\{r < \sigma < r + \delta\}} [d(\log \sigma \, d^c \log \sigma) - \log \sigma \, dd^c \log \sigma] \wedge \phi^*(\omega^{k-1}) \\
			  &	&	\\
			  & = & \frac{C \|\psi\|_{\infty}}{\delta} [\log(r+\delta) \; t_{k-1}(r+\delta) - \log r \; t_{k-1}(r)]  \\
			  &	&	\\
			  & 	& \;\;\;\; - \frac{C \|\psi\|_{\infty}}{\delta} \, \int_{\{r < \sigma < r+\delta\}} \log \sigma \; dd^c \log \sigma \wedge \phi^*(\omega^{k-1}) 
			  \end{array}
\end{equation}
We next examine the right hand term in the last line more closely: 
\begin{equation}
	\label{eqn:mvtint}
	\int_{\{r < \sigma < r+\delta\}} \log \sigma \, dd^c\log \sigma \wedge \phi^*(\omega^{k-1}) = \; \log(r + \alpha) [t_{k-1}(r+\delta) - t_{k-1}(r)],
\end{equation}
for some $\alpha \in (0, \delta)$, by the mean value theorem. Hence, resuming from (\ref{eqn:I2}) we get
\begin{equation}
	\label{eqn:I2bis}
	\begin{array}{rcl}
	T_k(r) \, I_2 & \leq & \frac{C \|\psi\|_{\infty}}{\delta} [\log(r+\delta) \, t_{k-1}(r+\delta) - \log r \; t_{k-1}(r)]  \\
		&	&	\\
		&	&\;\;\;\;  - \frac{C \|\psi\|_{\infty}}{\delta} [\log r (t_{k-1}(r+\delta) - t_{k-1}(r))] \\
		&	&	\\
		& \leq & \frac{C \|\psi\|_{\infty}}{\delta} [\log(1 + \frac{\delta}{r}) t_{k-1}(r+\delta)] \\
		&	&	\\
		& \leq &  C \|\psi\|_{\infty} \, \frac1{r} \,t_{k-1}(r+\delta),
	\end{array}
\end{equation}
where we have used $\log r < \log (r + \alpha)$ in the first line, and $\log (1 + y) \leq y, y > 0,$ in the last step. Since $\delta > 0$ was arbitrary, we conclude
\begin{equation}
	\label{eqn:I2fin}
	|I_2| \leq \frac1r \, C \|\psi\|_{\infty} \, \frac{t_{k-1}(r)}{T_k(r)}.
\end{equation}
Together, (\ref{eqn:I1}) and (\ref{eqn:I2fin}) show
\begin{equation}
	\label{eqn:vitesse}
	\frac1{c_r} |\left<dd^c S_r, \psi\right>| \leq C' \|\psi\|_{\infty} \, \frac{t_{k-1}(r)}{T_k(r)}.
\end{equation}
Applying this inequality gives the proof of the theorem.
\end{proof}

Before going on to analyze the $dd^c$-mass ratios $J_j(r)$, let us remark that one can also construct some $d$-closed cluster currents using the weight $u_r = \log^+\frac{r}{\sigma}$. 

\begin{theorem}
	\label{thm:logdclosed}
	Suppose 
	\begin{equation}
		\label{eqn:dlogcond}
		\liminf_{r \rightarrow R^-} \;\; \log r \, \frac{t_{j-1}(r) t_j(r)}{(T_j(r))^2} = 0.
	\end{equation}
	Then there exist positive d-closed cluster currents of mass 1 for $S_{j,r}/c_{j,r}$. Note that we use $u_r = \log^+ \frac{r}{\sigma}$ for the definition of the $S_{j,r}$.
\end{theorem}

\begin{proof}
We will just write out the case $j = k$. It suffices to estimate $\left<dS_r, \psi\right>$ with $\psi = \theta \wedge \beta^{k-1}$, as in (\ref{eqn:decomp}), where $\theta$ is a $(1,0)$-form and $\beta$ an arbitrary K\"ahler form, with bounds as in (\ref{eqn:decompbounds}). Then we have to estimate $\left<\bar{\partial}S_r, \theta \wedge \beta^{k-1}\right>$. As in the proof of (\ref{eqn:massratios}), we get
\begin{equation}
	\label{equation}
	\begin{array}{rcl}
	\frac1{c_r}|\left<\bar{\partial}S_r, \theta \wedge \beta^{k-1}\right>| & \leq &\frac1{c_r} \left(\int_{B_r} d\log \sigma \wedge d^c \log \sigma \wedge \phi^*(\beta^{k-1})\right)^{\frac12} \\
		&	& \;\;\;\;\;\; \times \left(\int_{B_r} \phi^*(\beta^{k-1}) \wedge \theta \wedge \bar{\theta})\right)^{\frac12}. 
	\end{array}
\end{equation}
The second term on the right is bounded by $C \|\psi\|_{\infty} t_k(r)^{\frac12}$. Squaring, we get 
$$\frac1{c_r^2} |\left<dS_r, \theta \wedge \beta^{k-1}\right>|^2 \leq C^2 \frac{\|\psi\|^2_{\infty}}{c_r^2} t_k(r) \int_{B_r} d\log \sigma \wedge d^c \log \sigma \wedge \phi^*(\beta^{k-1})$$
$$ = C^2 \frac{\|\psi\|^2_{\infty}}{c_r^2} \; t_k(r) \int_{B_r} \left(d(\log \sigma \, d^c \log \sigma) - \log \sigma \, dd^c\log \sigma\right) \wedge\phi^*(\beta^{k-1})$$
We have assumed for convenience that $\log \sigma \geq 0$, so this last becomes
\begin{equation}
	\label{eqn:findlog}
	\begin{array}{rcl}
	\frac1{c_r^2} |\left<dS_r, \theta \wedge \beta^{k-1}\right>|^2 & \leq & C^2 \frac{\|\psi\|^2_{\infty}}{c_r^2} \; t_k(r) \int_{\partial B_r} \log \sigma \, d^c \log \sigma \wedge\phi^*(\beta^{k-1}) \\
	&	&	\\
	& = & C^2 \frac{\|\psi\|^2_{\infty}}{c_r^2} \; t_k(r) \log r \int_{B_r}  dd^c\log \sigma \wedge\phi^*(\beta^{k-1}) \\
	&	&	\\
	& = & C^2 \|\psi\|^2_{\infty} \; \log r \; \frac{t_{k-1}(r) \, t_k(r)}{T_k^2(r)}.
	\end{array}
\end{equation}
Hence, we finally obtain
\begin{equation}
	\label{eqn:estdlog}
	\frac1{c_r}|\left<dS_r, \psi\right>| \leq C \|\psi\|_{\infty}\, \left(\log r \, \frac{t_{k-1}(r) \, t_k(r)}{T^2_k(r)}\right)^{\frac12},
\end{equation}
which proves the Theorem \ref{thm:logdclosed}
\end{proof}

Observe that when the exhaustion is bounded the term $\log r$ on the right of (\ref{eqn:estdlog}) disappears.

To analyze $J(t)$ in a fashion similar to that of $I(t)$ in equations (\ref{eqn:MR1dibis}) to (\ref{eqn:MR1supdelta}), we start, for $0 \leq j \leq k$, from
\begin{equation}
	\label{eqn:B}
		T_j(r) = \int_0^r \frac{t_j(s)}{s} \, ds.
\end{equation}
Then
\begin{equation}
	\label{eqn:Bprime}
		r\,T'_j(r) = t_j(r).
\end{equation}
Since the denominator of $J_j(r)$ is just $T_j(r)$, we can write $J_j(t)$, using (\ref{eqn:Bprime}), as follows:
\begin{equation}
	\label{eqn:tildeJmanip}
		J_j(r) = \frac{t_{j-1}(r)}{T_j(r)} = \frac{t_{j-1}}{t_j} \cdot \frac{t_j(r)}{T_j(r)} = \frac{t_{j-1}}{t_j} \cdot \frac{r T'_j(r)}{T_j(r)}.
\end{equation}
If there is no subsequence $r_{\ell} \rightarrow \infty$ such that $J_j(r_{\ell}) \rightarrow 0$, then there is a $c > 0$ such that $J_j(r) \geq c$ for all $r$. We have therefore 
\begin{equation}
	\label{eqn:MR2di}
	\frac{T'_j}{T_j} \geq \frac 1c \; \frac{t_j}{r \; t_{j-1}},
\end{equation}
which we integrate over the interval $[r_0, r]$ to obtain
\begin{equation}
	\label{eqn:MR2diint}
		\log T_j(s) ]_{r_0}^r \geq \frac 1c \int_{r_0}^r \frac{t_j(s)}{s\; t_{j-1}(s)} \, ds.
\end{equation}
We get that $J_j(r_{\ell}) \rightarrow 0$ for some subsequence $r_{\ell} \rightarrow R$ if 
\begin{equation}
	\label{eqn:MR2sup} 
		\limsup_{r < R} \frac{1}{\log T_j(r)} \cdot \int_{r_0}^r \frac{t_j(s)}{s \; t_{j-1}(s)} \, ds = +\infty, 
\end{equation}
provided, in the case that $R < \infty$, that $\log T_j(r) > 0$ for some $r \in [0, R]$. Then, arguing as in the proof of corollary \ref{cor:MR1delta}, we conclude the following corollary of Theorem \ref{thm:ddc-closed}. Note that the condition (\ref{eqn:MR2sup}) can be interpreted as saying that the relative growth of $\frac{t_j}{t_{j-1}}$ is large enough.

\begin{corollary}
	\label{cor:MR2sup}
		If (\ref{eqn:MR2sup}) holds, then $\omega, \phi, \sigma$ satisfy $dd^c$-$MR$.
\end{corollary}

If $R = +\infty$, we can draw some simple conclusions. If $k = \dim X = 1$, then 
\begin{equation}
	\label{eqn:MR2k1}
		J_1(r) = \frac{\int_{B_r} dd^c \log \sigma}{T_1(r)}.
\end{equation}
If $R = +\infty$, then $T_1(r) \gtrsim \log r$ as $r \rightarrow \infty$. If, in addition, $\sigma$ is a parabolic exhaustion of $X$, so that $\sigma$ is harmonic outside a compact subset of $X$, then by (\ref{eqn:MR2k1}) we get that any limit point $S_{\infty}$ of $S_r/c_r$ is a $dd^c$-closed positive current. If $\dim Y = 1$ also, then this must be a positive constant times the current $[Y]$ of integration on $Y$. 

As an illustration in a case where $R < +\infty$, consider $X = \mathbb{B}^1\subset \CC$, and $\sigma = |z|$. In this case, the condition that $\liminf_{r \rightarrow R} J(r) = 0$ is equivalent to 
\begin{equation}
	\label{eqn:k1R1cond}
		\int_{\mathbb{B}^1} (1-|z|) |\phi'(z)|^2 \, d\lambda(z) = +\infty,
\end{equation}
which can also be written as $\int_0^1 t_1(s) ds = +\infty.$ This condition was considered in \cite{fs} in connection with the study of laminations. For domain $\BB^k, k$ arbitrary, one would need the condition
$$\lim_{r \rightarrow 1^-} \frac{\int_{\BB^k_r} (1-\|z\|) \phi^*(\omega^k)}{\int_{\BB^k_r} (dd^c \log \|z\|) \wedge \phi^*(\omega^{k-1})} = +\infty.$$
\begin{remark}
	\label{rem:locddc}
	These last results may be localized. For example, if $\phi_n: \BB^k \rightarrow Y$ is a sequence of holomorphic maps such that 
	\begin{equation}
		\label{eqn:ddclocint}
		R(\phi_n) := \frac{\int_{\BB^k_1} (1 - \|z\|) \phi^*_n(\omega^k)}{\int_{\BB^k_1} (dd^c \log \|z\|) \wedge \phi_n^*(\omega^{k-1})} \rightarrow +\infty,
	\end{equation}
	as $n \rightarrow +\infty$, then the corresponding currents have among their clusterpoints a $dd^c$-closed positive current of bidemension $(k,k)$ and mass 1.
	
	It is interesting to compare the criteria obtained here and in Theorem \ref{thm:ddc-closed}. Let $\phi:\triangle \rightarrow Y$ be a holomorphic map. By Theorem \ref{thm:ddc-closed} we obtain a $dd^c$-closed current if 
	$$J(r) = \int_{|\zeta| < r} (r-|\zeta|)^+ |\phi'(\zeta)|^2 d\lambda(\zeta) \rightarrow +\infty,$$
	 while we get a $d$-closed current if 
	 $$\frac{\int_{|\zeta| < r} |\phi'(\zeta)|^2 d\lambda(\zeta)}{J(r)^2} \rightarrow 0.$$
\end{remark}
	
The techniques developed in this section may be applied to study the intersection of the $dd^c$-closed and positive currents constructed in this section with hypersurfaces in $Y$. Let $Z \subset Y$ be a hypersurface such that $\phi(X)$ is {\em not contained in} $Z$. Let $[Z]$ denote the current of integration over the hypersurface, and $\{Z\}$ the cohomology class of $Z$, of bidegree $(1,1)$. $\{T\}$ denotes the cohomology class of bidegree $(m-1,m-1)$ determined by $T$ of bidimension $(1,1)$. We use here that $Y$ is compact and K\"ahler. The $dd^c$-lemma on such varieties then gives the class $\{T\}$  by duality.

\begin{theorem}
	\label{thm:posint}
	Notation as above, we have
	
	\noindent 1.) if $X$ is parabolic, then 
	\begin{equation}
		\label{eqn:intpos1}
		\left<\{T\}, \{Z\}\right> \geq 0.
	\end{equation}
	
	\noindent 2. if $(dd^c \log \sigma)^k = 0$ outside a compact set in X, and
	\begin{equation}
		\label{eqn:T1cond}
		\lim_{r \rightarrow R} \frac{t_0(r)}{T_1(r)} = 0,
	\end{equation}
	then $\left<\{T\}, \{Z\}\right> \geq 0.$
\end{theorem}

\begin{proof} 

We have the equation of currents
\begin{equation}
	\label{eqn:PLformula}
	[Z] - \alpha = dd^c U,
\end{equation}
where $\alpha$ is a smooth $(1,1)$-form representing the class $\{Z\}$, and where $U$ can be assumed $\leq 0$ on $Y$, and $U\circ\phi$ is not identically $-\infty$. The pairing in the theorem is given by $\left<\{T\}, \{Z\}\right> := \left<T, \alpha\right> = \lim_{\ell \rightarrow \infty}\left<S_{r_{\ell}}, \alpha\right>/c_{r_{\ell}}.$
We now use the smoothings $\chi_{\delta}$ from the proof of Theorem \ref{thm:ddc-closed}, and set $u_{\delta, r_{\ell}} = \chi_{\delta}(\log \frac{r_{\ell}}{\sigma}).$ Note that $u_{\delta, r_{\ell}} \rightarrow u_{r_{\ell}}$ when $\delta \rightarrow 0$, and we set $\left<S_{\delta, r_{\ell}}, \alpha\right> = \int_{B_{r_{\ell}}} u_{\delta, r_{\ell}} (dd^c \log \sigma)^{k-1} \wedge \alpha. $ Thus, $\lim_{\delta \rightarrow 0^+} \left<S_{\delta, r_{\ell}}, \alpha\right> = \left<S_{r_{\ell}}, \alpha\right>.$  This said, we proceed to analyze $\left<S_{\delta, r_{\ell}}, \alpha \right>$:

\begin{equation}
	\label{eqn:poschain}
	\begin{array}{rcl}
	\left<S_{\delta,r_{\ell}}, \alpha\right> & = & \left<S_{\delta, r_{\ell}},  \alpha + dd^c U\right> - \left<S_{\delta, r_{\ell}}, dd^c U \right>\\
		&	&	\\
		& \geq & - \left<S_{\delta, r_{\ell}},  dd^cU \right>.
	\end{array}
\end{equation}
We have used here the obvious positivity inequality for the finite intersections
\begin{equation}
	\label{eqn:finitepart}
	\left<S_{\delta, r_{\ell}}, \alpha + dd^c U\right> = \int_{B_{r_{\ell}} \cap \phi^{-1}(Z)} u_{\delta, r_{\ell}}\; (dd^c\log \sigma)^{k-1} \geq 0,
\end{equation}
where $\phi^{-1}(Z)$ is counted with multiplicities. Now we use the fact that $u_{\delta, r_{\ell}}$ is compactly supported on $X$, and we integrate $dd^c$ by parts to get, as in the proof of Theorem \ref{thm:ddc-closed}: 
\begin{equation}
	\label{eqn:deltaIs}
	\begin{array}{rcl}
	- \left<S_{\delta, r_{\ell}}, dd^cU\right> & = & - \left< dd^c S_{\delta, r_{\ell}}, U\right> \\
		&	&	\\
		& = 	& \int_{B_{r_{\ell}}} \chi'_{\delta} (dd^c \log \sigma)^{k} \; U\circ\phi \\
		&	&	\\
		&	&	- \frac1{\delta}\int_{\{r < \sigma < r+\delta\}} d\log \sigma \wedge d^c\log \sigma \wedge (dd^c\log \sigma)^{k-1} \; U\circ\phi \\
		&	&	\\
		& = 	& I^{\delta, r_{\ell}}_1 + I^{\delta,r_{\ell}}_2.	
	\end{array}
\end{equation}

Now, $I^{\delta,r_{\ell}}_2 \geq 0,$ because $U \leq 0$, for any $\delta, r_{\ell}.$ As for $I^{\delta,r_{\ell}}_1$, suppose first that $X$ is parabolic. Then $(dd^c \log \sigma)^k$ is compactly supported on $X$, and we get, since $U \circ \phi$ is quasi-psh and hence locally integrable, 
$$\left<S_{r_{\ell}}, \alpha\right> \geq \int_{B_{r_{\ell}}} (dd^c\log \sigma)^k U \circ \phi \geq - C,$$
where $C$ is a positive constant, and then
$$\lim_{r_{\ell} \rightarrow \infty} \left<S_{r_{\ell}}, \alpha\right>/c_{r_{\ell}} \geq \lim_{r_{\ell} \rightarrow \infty} \frac{-C}{c_{r_{\ell}}} = 0,$$
since $c_{r_{\ell}} \rightarrow \infty$, for $X$ parabolic.

In the second case in the theorem, one still has $\int_X (dd^c \log \sigma)^k \; U\circ\phi$ bounded and  $c_{r_{\ell}} = T_1(r_{\ell}) \rightarrow \infty$.

\end{proof}

\begin{remark}
	\label{rmk:int}
	\noindent 1. When $X = \CC$ or a finite branched cover of $\CC$, the previous result is due to McQuillan \cite{mm}. It seems new even for the case $X$ parabolic of dimension 1. Note that, strictly speaking, McQuillan uses the average currents $\tilde{S}_r$ in remark \ref{rmk:flex}. The proof works the same in either case.

	\noindent 2. The result holds if we replace $\{Z\}$ by the class of a closed and positive current $R$ of bidimension $(n-1, n-1)$, provided we can write
	$$R = \alpha + dd^c U,$$
as above, where $\alpha$ is smooth, and $U\circ \phi$ is not identically $-\infty$.

	\noindent 3. If instead of a fixed map $\phi$, we suppose we have a sequence of maps $\phi_n: \triangle \rightarrow Y$ from the unit disk to $Y$.	Assume that there are sequences $n_{\ell}, r_{\ell}$ such that $n_{\ell} \rightarrow \infty,$ and $r_{\ell} \rightarrow 1^-$, and such that 
$$\lim_{\ell \rightarrow \infty} \; \frac{t_0(\phi_{n_{\ell}},r_{\ell})}{T_1(\phi_{n_{\ell}}, r_{\ell})} = 0,$$
and that $U \circ \phi_{n_{\ell}}$ does not converge to $-\infty$ uniformly. Then once again, any $dd^c$-closed cluster point $T$ of $S_{r_{\ell}}(\phi_{n_{\ell}})/c_{r_{\ell}}(\phi_{n_{\ell}})$ will verify $\left<T,\alpha\right> \geq 0.$ This is because we still have $\int_X (dd^c \log \sigma)^k \; U\circ\phi_{n_{\ell}}$ bounded.

	\noindent 4. If the hypersurface $Z$ is an ample divisor on $Y$, then we get $\left<\{T\},\{Z\}\right> = \left<T, \alpha\right> > 0,$ because we can take $\alpha$ to be a k\"ahler form on $Y$, and then $\left<T, \alpha\right>$ is just the mass of $T$ with respect to the Kaehler metric underlying $\alpha$. Similarly, if $\{Z\}$ is represented by a form $\alpha$ which is only non-negative, then $\left<\{T\}, \{Z\}\right> \geq 0,$ with equality if and only if the support of $T$ is contained in the zero locus of $\alpha$. It would be  interesting to know if there were other examples of geometric conclusions one could draw from the condition $\left<T, \alpha\right> = 0$.
\end{remark}


\section{Effect of scaling on the limits}
\label{sec:scaling}

In this section we want to change scales slightly when we compare the various volume measures we have discussed up to now. We will apply them to sequences of holomorphic maps $\phi_n$ with $X$ and $\tau$ fixed, using $dd^c$-closed limits in all intermediate dimensions. To this end, set $S_{j,n,r}(\psi) = \int_{B_r}  \log^+\frac{r}{\sigma} (dd^c \log \sigma)^{k-j} \wedge \phi_n^*(\psi),$ where $\psi$ is a test form on $Y$ of bidegree $(j,j)$, and set $c_{j,n,r} = S_{j,n,r}(\omega^j)$. Finally, set 
$$t_j(\phi_n, r) = \int_{B_r} (dd^c \log \sigma)^{k-j} \wedge \phi_n^*(\omega^j),$$
$$T_j(\phi_n, r) = \int_0^r t_j(\phi_n,r)\,\frac{ds}{s},$$
and
$$J_j(\phi_n, r) =  \frac{t_{j-1}(\phi_n,r)}{T_j(\phi_n, r)}.$$
Consider the condition that for some constant $c > 1$, we have that 
\begin{equation}
	\label{eqn:scalecond}
		\liminf_{n \rightarrow \infty, r \rightarrow R^-} \; \frac{t_{j-1}(\phi_n,r)}{t_j(\phi_n, r/c)} = 0.
\end{equation}
Note that this is similar to the condition in the hypotheses of corollary \ref{thm:egcase} below, except that here we are assuming that even a fixed fraction of the $t_j$ will dominate $t_{j-1}$, and considering a sequence of maps. 

\begin{theorem}
	\label{thm:scale}
	If condition (\ref{eqn:scalecond}) is verified, then there is a $dd^c$-closed positive cluster current of mass 1 for the family $\{S_{j,n,r}/c_{j,n,r}\}.$ 
\end{theorem}
\begin{proof}
We estimate $J_j(\phi_n, r)$ directly.
\begin{equation}
	\label{eqn:scalecomp1}
	\aligned
		J_j(\phi_n, r) & = \frac{\int_{B_r} (dd^c \log \sigma)^{k-j+1} \wedge \phi_n^{*}(\omega^{j-1})}{\int_{B_r} \log^{+} \frac r{\sigma} \; (dd^c \log \sigma)^{k-j} \wedge \phi_n^{*}(\omega^k)} \\
				& \leq  \frac{\int_{B_r} (dd^c \log \sigma)^{k-j+1} \wedge \phi_n^{*}(\omega^{j-1})}{\int_{B_{r/c}} \log^{+} \frac r{\sigma} \; (dd^c \log \sigma)^{k-j} \wedge \phi_n^{*}(\omega^j)} \\
				& \leq \frac{\int_{B_r} (dd^c \log \sigma)^{k-j+1} \wedge \phi_n^{*}(\omega^{j-1})}{\log c \; \int_{B_{r/c}} (dd^c \log \sigma)^{k-j} \;\phi_n^{*}(\omega^j)} \\
				& = \frac1{\log c} \cdot \frac{t_{j-1}(\phi_n, r)}{t_j (\phi_n, r/c)}.
	\endaligned
\end{equation} 
By (\ref{eqn:scalecond}) and the proof of Theorem \ref{thm:ddc-closed}, we get subsequences such that the $S_{j,n,r} / c_{j,n,r}$ converge to a $dd^c$-closed, positive current. 
\end{proof}

\begin{theorem}
	\label{thm:Brody}
		Let $\phi_n:\BB^k(1) \rightarrow Y$ be a sequence of holomorphic maps. Then either for some $j, 1\leq j \leq k$ there is a positive, $dd^c$-closed current $T$ which is a cluster point of $S_{j,n,r}/c_{j,n,r}$, or a subsequence of any sequence of graphs of the $\phi_n$ is convergent in the Hausdorff metric over any compact set in $\BB^k(1)$.
\end{theorem}

\begin{proof}
Suppose that, for $1 \leq j \leq k$, there are no such cluster currents. Then for any $j$ and for $n >> 0,$ and arbitrary $r < 1$, by Theorem \ref{thm:scale} and (\ref{eqn:scalecond}), we have for each such $j$ and any constant $c > 1$, but close to 1,
\begin{equation}
	\label{eqn:telescopingestimates}
	\int_{B_r} (dd^c\tau)^{k-j+1} \wedge \phi_n^*(\omega^{j-1}) \geq c_j \int_{B_{r/c}} (dd^c \tau)^{k-j} \wedge \phi_n^*(\omega^j),
\end{equation}
for $j = 1, \ldots, k,$ where the constant $c_j$ depends on the $c$ chosen. Telescoping gives, for every $r < 1 < c$ and each $j = 1,\ldots, k,$ independently of $j$ and $n >> 0$, a constant $C = C(c) > 0$ such that
\begin{equation}
	\label{eqn:Brodyvols}
	\int_{B_{r/c^k}} (dd^c \tau)^{k-j} \wedge \phi_n^*(\omega^j) \leq C,
\end{equation}
from which it follows that the volume of the graphs of $\phi_n$ over any fixed $B_{r/c^k}$ have a volume bound in $B_{r/c^k} \times Y$, independent of $n$. By Bishop's theorem, subsequences of the graphs then converge in the Hausdorff topology over any compact set $K \subset \subset \BB^k(1).$ By adjusting $c$, this convergence occurs over each compact set in $\BB^k(1)$.

\end{proof}


\section{Applications to value distribution}
	\label{sec:appsvd}
	
In this section we would like to apply some of the results above to classical value distribution. Some of the concepts above have clearly been motivated by this, and we start by recalling some of the classical definitions and results to make this explicit. We have $t_j(s) = \int_{B_s} (dd^c \tau)^{k-j} \wedge \phi^*(\omega^j),$ and $T_j(r) = \int_0^r t_j(s) \, \frac{ds}{s},$ as in definition \ref{def:tkdef} above. $T_j(r)$ is the characteristic function of order $j$. The classical case is when $\omega$ is the Chern form of an ample line bundle on a projective manifold $Y$. Recall the averaged counting function
\begin{equation}
	\label{eqn:ctgfn}
		N_{\phi}(D, r) = \int_0^r \frac{ds}{s} \int_{B_s \cap \phi^{-1}(D)} (dd^c \tau)^{k-1} := \int_0^r \frac{ds}{s} n_{\phi}(D, s).
\end{equation}
The First Main Theorem of value distribution for a hypersurface $D$ says that
\begin{equation}
	\label{eqn:fmt}
		N_{\phi}(D, r) + m_{\phi}(D, r) = T_1(r) + O(1),
\end{equation}
where the proximity function is given by
\begin{equation}
	\label{eqn:prox}
		m_{\phi}(D, r) = \int_{\partial B_r} \log \frac{\|\zeta\|}{|\zeta \circ \phi|} d^c \tau \wedge (dd^c \tau)^{k-1} \geq 0.
\end{equation}
Here $\zeta$ is a section of $L = L(D)$ on $Y$ such that $\zeta^{-1}(0) = D$,  $|\zeta|$ is a point-wise norm for sections of $L$ on $Y$, and $\|\zeta\|$ is a corresponding global norm on the space of global sections of the line bundle $L$, say by integrating the point-wise norm $|\zeta|$. By (\ref{eqn:prox}), the FMT says
$$N_{\phi}(r) \leq T_{\phi, 1}(r) + O(1).$$

For simplicity and explicitness, let us first consider more closely the case of hyperplanes $D$ in $\PP^m = \PP(\CC^{m+1})$. So we let $a \in \check{\PP}^m = \PP(\check{\CC}^{m+1})$, the dual projective space, and recall the Poincar\'e-Lelong formula of currents
\begin{equation}
	\label{eqn:plf}
		dd^c \log \frac{\|z\| \|a\|}{|\left<z, a\right>|} = \omega - [D_a],
\end{equation}
where $\omega$ is the Fubini-Study form, the first Chern form of $L = L(D_a)$.

Now suppose we can choose a probability measure $\nu$ on $\check{\PP}^m$ such that
\begin{equation}
	\label{eqn:nupot}
	U_{\nu}(z) := \int \log \frac{\|z\| \|a\|}{|\left<z, a\right>|} \, d\nu(a) \leq C < +\infty.
\end{equation}
Such measures can be supported on very small sets, for example, any set of positive Lebesgue measure on a real analytic arc in $\check{\PP}^m$ not contained in a hyperplane, or supported on any non-pluripolar set, cf. \cite{mss}. 

Given such a measure $\nu$, we can state a precise theorem in this context. (The definition of positive capacity is reviewed below, in (\ref{eqn:defcap}))

\begin{theorem}
	\label{thm:defvit}
	Let $\phi$ be a holomorphic map $\phi: X \rightarrow \PP^m$. Let $\mathcal{E}$ be a set of hyperplanes $D_a \subset \PP^m$ of positive capacity with respect to the kernel $K(z,a) = \log \frac{\|z\| \, \|a\|}{|\left<z,a\right>|}$. Then
	\begin{equation}
		\label{eqn:defvit}
		| 1 - \int_{\check{\PP^m}} \frac{N(D_a, r)}{T_1(r)} d\nu(a)| \leq C \|U_{\nu}||_{\infty} \frac{t_0(r)}{T_1(r)}.
	\end{equation}
\end{theorem}

\begin{proof}
Consider the bounded function $U_{\nu}$ of equation \ref{eqn:nupot}. We get
\begin{equation}
	\label{eqn:pairUnu}
\begin{array}{rcl}
\left<dd^c \frac{S_{r}}{T_1(r)}, U_{\nu}\right> & = & \left<\frac{S_{r}}{T_1(r)}, dd^c U_{\nu}\right> \\
	&	&	\\
	& = & \frac1{T_1(r)} \left<S_{r}, \omega - \int_{\check{\PP}^m} [D_a] \, d\nu(a)\right> \\
	&	&	\\
	& = & 1 - \frac1{T_1(r)} \int_{\check{\PP}^m} \left<S_{r}, [D_a]\right> \, d\nu(a) \\
	&	&	\\
	& = & 1 - \frac1{T_1(r)} \int_{\check{\PP}^m} \{\int_{B_r \cap \phi^{-1}(D_a)} \log^+ \frac{r}{|z|} \, (dd^c \tau)^{k-1} \} d\nu(a) \\
	&	&	\\
	& = & 1 - \frac1{T_1(r)} \int_0^{r} \frac{ds}{s} \int_{\check{\PP}^m} \{ \int_{B_s \cap \phi^{-1}(D_a)} (dd^c \tau)^{k-1} \} \, d\nu(a) \\
	&	&	\\
	& = & 1 - \frac1{T_1(r)} \int_0^{r} \frac{ds}{s} \int_{\check{\PP}^m} n(D_a, s) \; d\nu(a).
\end{array}
\end{equation}
Since $U_{\nu}$ is bounded, by Theorem \ref{thm:ddc-closed} we get
\begin{equation}
	\label{eqn:ddcestU}
	|\left<dd^c \frac{S_{r}}{T_1(r)}, U_{\nu}\right>| \leq C \|U_{\nu}\|_{\infty} \frac{t_0(r)}{T_1(r)}
\end{equation}
and by (\ref{eqn:pairUnu}), we get
\begin{equation}
	\label{eqn:defest}
	|1 - \frac1{T_1(r)} \int_0^{r} \frac{ds}{s} \left[\int_{\check{\PP}^m} n(D_a, s) d\nu(a)\right]| \leq C \|U_{\nu}\|_{\infty} \frac{t_0(r)}{T_1(r)},
\end{equation}
which was to be proved.
\end{proof}

In the case that we can guarantee the existence of a $dd^c$-closed cluster point of $S_r/c_r$, namely, $\liminf_{r\rightarrow R} \frac{t_0(r)}{T_1(r)} = 0,$ we get that the left hand side of (\ref{eqn:defest}) goes to 0 along a subsequence $r_{\ell} \rightarrow R^-$, and hence,
\begin{equation}
	\label{eqn:avratio}
	\lim_{\ell \rightarrow \infty} \frac{\int_0^{r_{\ell}} \frac{ds}{s} \left[\int_{\check{\PP}^m} n(D_a, s) \, d\nu(a)\right]}{T_1(r_{\ell})} = \lim_{j\rightarrow \infty} \frac{\int_{\check{\PP}^m} N(D_a, r_{\ell}) \, d\nu(a) }{T_1(r_{\ell})} = 1.
\end{equation}
Thus, we have that 
\begin{equation}
	\label{eqn:def0div}
	\limsup_{\ell \rightarrow +\infty} \frac{N(D_a, r_{\ell})}{T_1(r_{\ell})} = 1,
\end{equation}
for $\nu$-almost every point $a$ in the support of $\nu$. Thus the exceptional set of $a$ for which (\ref{eqn:def0div}) does not  hold must be a set $\mathcal{E}$ of capacity 0 for the kernels $K(z, a) = \log\frac{\|z\| \|a\|}{|\left<z, a\right>|}$, that is, $\mathcal{E}$ does not carry a probability measure $\mu$ for which $U_{\nu}$ in (\ref{eqn:nupot}) is bounded. In particular, as already noted, a non-pluripolar set $E$ is too large to be exceptional in this sense, cf. \cite{mss}.

Now let us consider defect relations such as (\ref{eqn:def0div}) for dimensions other than $k-1$, i.e., for $D$ of dimension other than $m-1$. The cases different from $D$ a divisor are all formally similar, and not as precise as the case of divisors $D$ above. The most interesting is the case of points, i.e., where we consider a non-degenerate holomorphic map $\phi: X \rightarrow \PP^m, m \geq k = \dim X,$ and we let $D_a \subset \PP^m$ be a linear subspace of $\dim D_a = m-k$, where $a$ is parametrized by the Grassmannian $Gr := Gr(m+1, m-k+1)$. We will consider this case in what follows.

We consider a potential $U_a$, i.e., a $(k-1,k-1)$-form on $\PP^m$ with integrable coefficients, satisfying the following analogue of the Poincar\'e-Lelong formula
\begin{equation}
	\label{eqn:PLhighercod}
	dd^c U_a = \omega^k - [D_a],
\end{equation}
where we take $\omega$ to be the normalized Fubini-Study class which gives an integral generator of $H^2(\PP^m, \ZZ)$. We can choose $U_a \geq 0$, and is obtained as 
\begin{equation}
	\label{eqn:Uadef}
	U_a = \left<D_a(\zeta), K(z,\zeta)\right>,
\end{equation}
where the singularity of the kernel can be bounded by $|\log|z-\zeta| | \cdot |z-\zeta|^{-2k+2},$ see Dinh-Sibony \cite{ds2} for a detailed estimate of the kernel. We introduce a capacity $C_k$ on $Gr$ as follows. For a probability measure $\nu$ on $Gr$, set 
\begin{equation}
	\label{eqn:Unuhcod}
	U_{\nu}(z) = \int_{Gr} U_a(z) \, d\nu(a).
\end{equation}
Define $\sup U_{\nu}(z), z \in \PP^m,$ as the infimum of all $C > 0$ such that $U_{\nu}(z) \leq C \omega^{k-1}(z)$, and let $\|U_{\nu}\|_{\infty} = \sup_{z \in \PP^m} \sup U_{\nu}(z)$. Let 
\begin{equation}
	\label{eqn:defcap}
	C(A) = \sup_{\nu \in \mathcal{M}(A)} \left(\frac1{\|U_{\nu}\|_{\infty}}\right),
\end{equation}
where $\mathcal{M}(A)$ is the space of probability measures supported on $A \subset \PP^m$. It turns out that $C(A) > 0$ if and only if there is a probability measure $\nu$ on $A$ such that $U_{\nu}(z) \leq 2\, C(A)$, for every $z \in \PP^m$, independently of $z$, see \cite{ds2}. For example, if $m = k$, it is enough that
$$\int_{Gr} \frac{|\log |z-a||}{|z-a|^{2k-2}} \, d\nu(a) \leq C,$$
so, if the Hausdorff dimension of $A$ is strictly larger than $2k-2$, one can construct such a measure, and $C(A) > 0$. See \cite{jc}.

If we define, similar to the classical case, 
\begin{equation}
	\label{eqn:NDa}
	N(D_a, r) = \int_0^r \frac{ds}{s} \, \int_{B_s} \phi^{-1}([D_a]) = \int_{B_r \cap \phi^{-1}(D_a)} u_r(z), 
\end{equation}
and
\begin{equation}
	\label{eqn:delDar}
	\delta(D_a, r) = 1 - \frac{N(D_a,r)}{T_k(r)},
\end{equation}
and finally the defect
\begin{equation}
	\label{eqn:defdef}
	\delta(D_a) = \limsup_{\; r \rightarrow R^-} \; \delta(D_a, r),
\end{equation}
then we have the following analogue of theorem \ref{thm:defvit}.
\begin{theorem}
	\label{thm:defvit0}
	Let $\phi: X \rightarrow \PP^m$ be as above. Let $\nu$ be a probability measure on $Gr$ such that $U_{\nu} := \int_{Gr} U_a \, d\nu(a)$ has bounded coefficients, $\|U_{\nu}\|_{\infty} < C < \infty$, then 
	\begin{equation}
		\label{eqn:adefrel}
		\int_{Gr} |\delta(D_a, r)| \, d\nu(a) \leq C' \|U_{\nu}\|_{\infty} \frac{t_{k-1}(r)}{T_k(r)}.
	\end{equation}
In particular, $\delta(D_a) = 0$, for $\nu$-a.e. a, if  $\limsup_{\; r \rightarrow R^-} \, \frac{t_{k-1}(r)}{T_k(r)} = 0,$ and
\begin{equation}
	\label{eqn:nuest}
	\nu(\{a \, | \, \delta(D_a,r) > \epsilon\}) \leq \frac{C'}{\epsilon} \; \frac{t_{k-1}(r)}{T_k(r)}.
\end{equation} 
\end{theorem}

\begin{proof}
We have 
$$dd^cU_a = \omega^k - [D_a].$$
Thus
$$1 - \frac{N(a,r)}{T_k(r)} = \frac1{c_r}[\left<S_r, \omega^k\right> - \left<S_r, [D_a]\right>] = \frac1{c_r} \left<dd^c S_r, U_a\right>.$$
We integrate this last relation with respect to $\nu$ and get
$$1 - \frac{\int_{Gr} N(a, r) \, d\nu(a)}{T_k(r)} = \frac1{c_r} \left< dd^c S_r, U_{\nu}\right>.$$
Using Theorem \ref{thm:ddc-closed}, we get the defect estimate
\begin{equation}
	\label{eqn:defestanu}
	\int_{Gr} |\delta(a, r)| \, d\nu(a) \leq C \, \|U_{\nu}\|_{\infty} \, \frac{t_{k-1}(r)}{T_k(r)},
\end{equation}
which proves all the claims of the theorem.
\end{proof}

\begin{remark} 

\noindent 1. When $k=m$, we get in particular that if $\liminf_{\; r \rightarrow R^-} \, \frac{t_{k-1}(r)}{T_k(r)} = 0,$ then the map $\phi$ omits a set of Hausdorff measure $\leq 2k - 2 + \epsilon$, for any $\epsilon > 0.$

\noindent 2. Instead of a fixed map, we can consider a sequence of maps $\phi_n: X \rightarrow \PP^m$ of holomorphic, non-degenerate maps. If 
$$\lim_{\ell \rightarrow \infty} \, \frac{t_{k-1}(\phi_{n_{\ell}}, r_{\ell})}{T_k(\phi_{n_{\ell}}, r_{\ell})} = 0,$$
\emph{cf.} the similar comment in remark \ref{rmk:int}, part 3., for the notation. Then we get an estimate
$$|\int_{Gr} \delta(\phi_{n_{\ell}}, D_a, r_{\ell}) \, d\nu(a) | \leq C \|U_{\nu}\|_{\infty} \, \frac{t_{k-1}(r_{\ell})}{T_k(r_{\ell})} \rightarrow 0.$$

\noindent 3. The potentials $U_{\nu}$ in (\ref{eqn:nupot}) and (\ref{eqn:Unuhcod}) play the role here of the proximity function in the classical theory. One might refer to them as {\em proximity potentials}.

\end{remark}

We close this section with a corollary on the behavior of holomorphic foliations by Riemann surfaces.

\begin{corollary}
	\label{cor:fol}
	Let $\mathcal{F}$ be a holomorphic foliation of $\PP^m$ by Riemann surfaces with finitely many singularities. Assume that all singularities are hyperbolic, and that there are no algebraic (compact) leaves. Fix a leaf $L$. There is a pluripolar set $\mathcal{E}_L \subset \check{\PP}^m$ such that for $a \notin \mathcal{E}_L$ the corresponding hyperplane $D_a$ intersects $L$ infinitely many times with the estimate given by Theorem \ref{thm:defvit0}.
\end{corollary}

\begin{proof}
The assumptions imply that all leaves are uniformized by the unit disk $\triangle$ \cite{fs}. It is further shown in \cite{fs} that if $\phi: \triangle \rightarrow L$ is the universal covering, then 
$$\int_{\triangle} (1- |z|) |\phi'(z)|^2 d\lambda(z) = +\infty,$$
where $\lambda$ is Lebesgue measure on $\triangle$. Thus, for the  map $\phi$ and exhaustion of $\triangle$ given by $|z|^2$, we have 
$$\lim_{r \rightarrow 1^-} \frac{t_0(r)}{T_1(r)} = 0,$$
and we can apply Theorem \ref{thm:defvit0}.

\end{proof}


\section{Equidistribution results in higher dimensions}
\label{sec:hdims}

In this section we would like to consider some equidistribution results for maps $\phi: X \rightarrow Y$, where $\dim Y = m > k = \dim X$. For example, we might have a birational map $f: Y \rightarrow Y,$ and $\phi: \CC^m \rightarrow Y$ parametrizes some stable manifold associated with $f$, e.g., the stable manifold of a periodic point of $f$, or a Pesin stable manifold ({\em cf.} \cite{kh}, for example).

We give a specific example from dynamics. Let $f:\CC^m \rightarrow \CC^m$ be a polynomial automorphism, and denote also by $f$ its extension $\PP^m \cdots \rightarrow \PP^m$ as a birational map. Let $I_{\pm}$ be the indeterminacy sets of $f, f^{-1}$, respectively, in the hyperplane at infinity of $\PP^m$. One calls $f$ {\em regular} if $I_+ \cap I_- = \phi,$ in which case we have an integer $p$ such that $\dim I_+ = m-p-1,$ and $\dim I_- = p-1.$ Let 
$$K_+ = \{ z \in \CC^m \, | \, \{f^n(z) | n \in \NN\} \, \text{is bounded} \subset \CC^m\}.$$ 
Then $K_+$ is closed in $\CC^m$, and $\bar{K}_+ \subset \PP^m = K_+ \cup I_+.$ Furthermore, if $\deg f = d_+, \deg f^{-1} = d_-,$ then $d_+^p = d_-^{k-p}$. Finally, define 
\begin{equation}
	\label{eqn:tpp}
		T_+ = \lim \frac{(f^n)^* \omega}{d_+^n}.
\end{equation}
Then we recall from \cite{ds1} the following theorem.
\begin{theorem} (Dinh-Sibony) 
	\label{thm:ds1}
	$T^p_+$ is the unique closed positive current of bidimension (p,p) and mass 1 supported on $\bar{K}_+$.
\end{theorem}
Note also the following corollary of theorem \ref{thm:ds1} from \cite{ds1}.
\begin{corollary} (Dinh-Sibony) 
	\label{cor:ds1} 
		If p = m-1, and $\phi: X \rightarrow \bar{K}_+ \subset \PP^m$, with X a parabolic Riemann surface ($k = 1$), then the image of X is dense in $\bar{K}_+$. In fact, all the closed cluster currents $(S_{1,r}/c_{1,r})$ of Corollary \ref{cor:k1par} coincide with $T_+^{m-1}$.
\end{corollary}

In particular the automorphism $f$ can have an attractive fixed point $z_0 \in \CC^m.$ The domain of attraction $U(z_0)$ is then biholomorpic to $\CC^m$ and is contained in $K_+$. It is called a Fatou-Bieberbach domain. Clearly it is not dense in $\CC^m$. Moreover it follows from the previous results that any positive closed current of bidimension $(1,1)$ constructed as in this paper using images of a parabolic manifold $X$, by any holomorphic map $\phi: X \rightarrow \overline{K}_+ \subset Y = \PP^m$ in any dimension $1 \leq k \leq m$ by taking limit points of the currents $S_{1,r}/S_{1,r}(\omega)$ will be equal to a multiple of $T^p_-$. That is, for all such $\phi, X$, one has $\mathcal{C}_1(\phi) = \{T^p_+\}$ ({\em cf.}, definition \ref{def:spacedclosed}).


\section{Examples: growth conditions}
\label{sec:growth}

We give here some simple examples of the theorems above, compared both to the usual growth conditions of the theory of entire functions. Let us first fix the terminology.
\begin{definition}
	\label{def:finord}
	The map $\phi$ is of exponential growth (or of finite order) if
	$$t_k(r) \lesssim r^d, \;\text{some}\; d, \text{as}\; r \rightarrow +\infty.$$
\end{definition}
\noindent Here we use the unaveraged order function $t_k(r)$ for the $dd^c$ case, 
$$t_j(r) = \int_{B_r} (dd^c \log \sigma)^{k-j} \wedge \phi^*(\omega^j)$$
in the case $j= k$, cf. (\ref{def:ddcMR}) and following. For convenience let us define $\mathcal{H}_k(\phi) = \{dd^c$-closed limit currents of $S_r/c_r\}$.
\begin{theorem}
	\label{thm:egcase}
	Suppose $\phi$ of exponential growth, and
	$$\frac{t_k(r)}{t_{k-1}(r)} \rightarrow \infty,$$	
	as $r \rightarrow +\infty.$ Then $\mathcal{H}_k(\phi)$ is non-empty.
\end{theorem}
\begin{proof}
Under these hypotheses, $T_k(r) = \int_0^r t_k(s) \frac{ds}{s} \lesssim r^d$, and hence $\log T_k(r) \lesssim d \log r.$ In this case, then, we can say
$$\frac{1}{\log T_k(r)} \int_{r_0}^r \frac{t_k(s)}{t_{k-1}(s)} \frac{ds}{s}  \gtrsim \frac1{d \log r} \int_{r_0}^r \frac{t_k(s)}{t_{k-1}(s)} \frac{ds}{s} \rightarrow +\infty,$$
as $r \rightarrow +\infty$. Taking note of corollary \ref{cor:MR2sup}, this proves the theorem.
\end{proof}

Another example is given by another, slower order of growth. 

\begin{theorem}
	\label{thm:explogp}
	If $t_k(t) \lesssim (\log t)^p,$ and 
	\begin{equation}
		\label{eqn:logorder}
		\frac{t_k(r)}{t_{k-1}(r)} \geq \frac{c}{(\log t)^{\beta}}, \; \beta < 1, 0 < c,
	\end{equation}
	then $\mathcal{H}_k(\phi)$ is non-empty.
\end{theorem} 
\begin{proof}
	Under these hypotheses, we have
	$$\log T_k(r) \lesssim (p+1) \log \log r.$$
	Integrating (\ref{eqn:logorder}) against $\frac{ds}{s}$, we get 
	$$\int_{r_0}^{r} \frac{t_k(s)}{t_{k-1}(s)} \frac{ds}{s} \geq c \int_{r_0}^r \frac1{(\log s)^{\beta}} \frac{ds}{s} \sim \frac{c}{1-\beta} (\log r)^{(1-\beta)} $$
	Diving both sides by $\log T_k(r)$, we see that as $r \rightarrow +\infty$, we get
	$$\lim_{r \rightarrow +\infty} \frac{1}{\log T_k(r)} \int_{r_0}^r  \frac{t_k(s)}{t_{k-1}(s)} \frac{ds}{s} \, \gtrsim  \frac{c}{(1-\beta)(p+1)} \frac{(\log r)^{1-\beta}}{\log \log r} \rightarrow +\infty,$$
	i.e., condition (\ref{eqn:MR2sup}). By corollary \ref{cor:MR2sup}, we conclude $\mathcal{H}_k(\phi)$ is non-empty.
\end{proof}


\end{document}